\documentclass[a4paper,leqno]{amsart}

\usepackage{amsbsy,amscd,amsmath,amsopn,amssymb,amsthm}
\usepackage{color} 

\theoremstyle{definition}
\newtheorem{thm}{Theorem}[section]
\newtheorem{lem}[thm]{Lemma}
\newtheorem{prp}[thm]{Proposition}
\newtheorem{dfn}[thm]{Definition}
\newtheorem{cor}[thm]{Corollary}

\newtheorem{rmk}[thm]{Remark}

\newtheorem{exa}[thm]{Example}

\newcommand{\Q}{{\mathbf{Q}}}

\newcommand{\R}{{\mathbf{R}}}
\newcommand{\C}{{\mathbf{C}}}
\newcommand{\N}{{\mathbf{N}}}

\begin{document}
\title[matrix monotone functions and matrix convex functions]
{Double piling structure of matrix monotone functions and of matrix convex functions II}

\author[H. Osaka]{HIROYUKI OSAKA$^{a}$} 
\date{13, Apr., 2011}
\thanks{}
\address{Department of Mathematical Sciences, Ritsumeikan
University, Kusatsu, Shiga 525-8577, Japan}

\email{osaka@se.ritsumei.ac.jp}

\author[J. Tomiyama]{JUN TOMIYAMA}
\thanks{}
\address{Prof. Emeritus of Tokyo Metropolitan University,
201 11-10 Nakane 1-chome, Meguro-ku, Tokyo, Japan}

\email{juntomi@med.email.ne.jp}



\footnote{$^A$Research partially supported by the JSPS grant for 
Scientic Research No. 20540220.}

\maketitle

\begin{abstract}
We continue the analysis in [H. Osaka and J. Tomiyama,
Double piling structure of matrix monotone functions and of matrix convex functions, 
Linear and its Applications {\textbf{431}}(2009), 1825 - 1832] 
in which the followings three assertions at each label $n$ are discussed: 
\begin{enumerate}
\item $f(0) \leq 0$ and $f$ is $n$-convex in $[0, \alpha)$
\item For each matrix $a$ with its spectrum in $[0, \alpha)$ and a contraction $c$ in the matrix algebra $M_n$, 
$$
f(c^*ac) \leq c^*f(a)c.
$$
\item The function $f(t)/t$ $(= g(t))$ is $n$-monotone in $(0, \alpha)$.
\end{enumerate}
We know that two conditions $(2)$ and $(3)$ are equivalent and if $f$ with $f(0) \leq 0$
 is $n$-convex, then $g$ is $(n -1)$-monotone. In this note we consider several extra conditions 
on $g$ to conclude that the implication from $(3)$ to $(1)$ is true. 
In particular, 
we study a class $Q_n([0, \alpha))$ of functions with conditional positive Lowner matrix 
which contains the class of matrix $n$-monotone functions and show that if 
$f \in Q_{n+1}([0, \alpha))$ with 
$f(0) = 0$ and $g$ is $n$-monotone, then $f$ is $n$-convex. 
We also discuss about the local property of $n$-convexity.
\end{abstract}

\section{Introduction}

Let $n \in \N$ and $M_n$ be the algebra of $n \times n$ matrices. 
We call a function $f$ matrix convex of order $n$ or 
$n$-convex in short whenever the inequality 
$$
f(\lambda A + (1 - \lambda)B) \leq \lambda f(A) + (1 - \lambda) f(B), \ 
 \lambda \in [0, 1]
 $$ holds for every pair of selfadjoint matrices $A, B \in M_n$ such that 
all eigenvalues of $A$ and $B$ are contained in $I$.
Matrix monotone functions on $I$ are 
similarly defined as
the inequality 
$$
A \leq B \Longrightarrow f(A) \leq f(B)
 $$
for an arbitrary selfadjoint matrices $A, B \in M_n$ such that 
$A \leq B$ and  all eigenvalues of $A$ and $B$ are contained in $I$.

We denote the spaces of operator monotone functions 
and of operator convex functions by $P_\infty(I)$ and $K_\infty(I)$ 
respectively. The spaces for $n$-monotone functions and $n$-convex functions 
are written as $P_n(I)$ and $K_n(I)$. 
We have then
\begin{align*}
&P_1(I) \supseteq \cdots \supseteq P_{n-1}(I) \supseteq P_n(I) \supseteq P_{n+1}(I) \supseteq \cdots 
\supseteq P_\infty(I)\\
&K_1(I) \supseteq \cdots \supseteq K_{n-1}(I) \supseteq K_n(I) \supseteq K_{n+1}(I) \supseteq \cdots 
\supseteq K_\infty(I)\\
\end{align*}
Here we meet the facts that $\cap_{n=1}^\infty P_n(I) = P_\infty(I)$ 
and $\cap_{n=1}^\infty K_n(I) = K_\infty(I)$. 
We regard these two decreasing sequences
 as noncommutative counterpart of the classical piling 
sequence $\{C^n(I), C^\infty(I), Anal(I)\}$, where the class $Anal(I)$ 
denotes the set of all of analytic functions over $I$.
We could understand that the class of operator monotone functions $P_\infty(I)$ 
corresponds to the class $\{C^\infty(I), Anal(I)\}$ by the famous characterization of those 
functions by Loewner as the restriction of Pick functions.
 
In these circumstances, it will be well recognized that we should not stick our discussions 
only to those classes $P_\infty(I)$ and $K_\infty(I)$, that is, the class of operator monotone 
functions and that of operator convex functions. Those classes $\{P_n(I)\}$ and $\{K_n(I)\}$
are not merely optional ones to $P_\infty(I)$ and $K_\infty(I)$. 
They should play important roles in the aspect of noncommutative calculus as the ones 
$\{C^n(I)\}$ play in usual (commutative) calculus.

The first basic question is whether $P_{n+1}(I)$ $($resp. $K_{n+1}(I))$ is 
strictly contained in $P_n(I)$ $($resp. $K_n(I))$ for every $n$.
This gap problem for arbitrary $n$ has been solved only recently (\cite{HGT}, \cite{OST}, \cite{HT}).
 
On the other hand, there are basic equivalent assertions known only at the 
level of operator monotone functions and operator convex functions by \cite{HP}, \cite{HP2}.
We shall discuss those (equivalent) assertions as the correlation problem between two kinds of piling structures 
$\{P_n(I)\}$ and $\{K_n(I)\}$, that is, we are planning to discuss relations between those 
assertions at each level $n$.

In \cite{OJ} we discussed about  the following 3 assertions at each level $n$ among them 
in order to see clear insight of the aspect of the problems:
\begin{enumerate}
\item[(i)] $f(0)\leq 0$ and $f$ is $n$-convex in $[0,\alpha)$,
\item[(ii)] For each matrix $a$ with its spectrum in $[0,\alpha)$ and a contraction $c$ in the matrix algebra
$M_n$,
 \[f(c^{\star}a c)\leq c^{\star}f(a)c,\]
\item[(iii)] The function $\frac{f(t)}{t}$ $(= g(t))$ is $n$-monotone in $(0,\alpha)$.
\end{enumerate}
Then we showed that  for each $n$ the condition (ii) 
is equivalent to the condition (iii) and  the assertion 
that  $f$ is $n$-convex with $f(0) \leq 0$ implies that $g(t)$ is 
$(n-1)$-monotone holds. 

In this note we continue to consider the double piling structure in \cite{OJ} and 
focus  our discussion to the class $Q_n(I)$ of all real $C^1$ 
functions $f$ on the interval $I$ such that 
for each $\lambda_1, \lambda_2, \dots, \lambda_n \in I$ 
the corresponding Lowner matrix $([\lambda_i, \lambda_j]_f)$
is an almost positive matrix (\cite[IV]{Dg}). 
We discuss about the relation between the condition (i)
and the condition (iii). 
For example, we show that if $f$ is $2$-convex on $[0, \alpha)$, then 
$g(t) (= \frac{f(t)}{t})$ belongs to $Q_2(0, \alpha)$. Conversely, 
when $f \in Q_n([0, \alpha))$ with $f(0) = 0$, 
if $g(t)$ is (n - 1)-monotone, then $f$ is $(n - 1)$-convex. 
We note that for each $n \in \N$ the class $P_n(I)$ is a subset of $Q_n(I)$ 
and $K_n(I) \cap Q_n(I) \not=\emptyset$. It could be that $P_n(I)$ is a proper 
subset of $Q_n(I)$. 
In fact we know that $t^2 \in Q_2(0, \alpha)\backslash P_2(0, \alpha)$ for any 
$\alpha > 0$. 
On the contrary, for $0 < \alpha < 1$, $n \geq 3$ and $\beta > 0$, then 
$t^\alpha \in Q_n([0, \beta))$, but $t^\alpha \not\in K_{n-1}([0, \beta))$.

The authors are indebted to a recent work \cite{FS} by F. Hiai and T. Sano for giving an attension 
to the class $Q_n(I)$.


\section{Preliminary}

We shall sometimes use the standard regularization procedure, 
cf. for example Donoghu \cite[p11]{Dg}. 
Let $\phi$ be a positive and even $C^\infty$-function defined on the real axis, vanishing
ooutside the closed interval $[-1, 1]$ and normalized such that
$$
\int_{-1}^1\phi(x) = 1.
$$
For any locally intergrable function $f$ defined in an open interval 
$(a, b)$ we form its regularization 
$$
f_\varepsilon(t) = 
\frac{1}{\varepsilon}\int_a^b\phi(\frac{t-s}{\varepsilon})
f(s)ds, \quad t \in \R
$$
for small $\varepsilon > 0$, and realize that it is infinitely 
many times differentiable. 
For $t \in (a + \varepsilon, b - \varepsilon)$ we may also write
$$
f_\varepsilon(t) = \int_{-1}^1\phi(s)f(t - \varepsilon s)ds.
$$
If $f$ is continuous, then $f_\varepsilon$ converges uniformly to 
$f$ on any compact subinterval of $(a, b)$.  
If in addition $f$ is $n$-convex (or $n$-monotone) in $(a, b)$, then 
$f_{\varepsilon}$ is $n$-convex (or $n$-monotone) in the slightly smaller 
interval $(a + \varepsilon, b - \varepsilon)$. Since the pointwise
 limit of a sequence of $n$-convex (or $n$-monotone) functions 
is again $n$-convex (or $n$-monotone), we may therefore in many 
applications assume that an $n$-convex or $n$-monotone function 
is sufficiently many times differentiable. 

For a sufficiently smooth function $f(t)$ we denote its n-th divided difference 
for n-tuple of points $\{t_1,t_2,\ldots,t_n\}$ defined as, when they are all 
different,
\[
[t_1,t_2]_f = \frac{f(t_1) - f(t_2)}{t_1 - t_2} ,\mbox{and inductively}
\]
\[
[t_1,t_2,\ldots,t_n]_f = \frac{[t_1,t_2,\ldots,t_{n-1}]_f - [t_2,t_3,\ldots,t_n]_f}{t_1 - t_n}.
\]
And when some of them coincides such as $t_1 = t_2 $ and so on, we put as 
\[[t_1,t_1]_f = f'(t_1)\quad \mbox { and inductively}
\]
\[
 [t_1,t_1,\ldots,t_1]_f = \frac{f^{(n)}(t_1)}{n!}.
\]

When there appears no confusion we often skip the referring function $f$. 
We notice here the most important property of divided differences
is that it is 
free from permutations of $\{t_1,t_2,\ldots,t_n\}$ in an open interval $I$.

\begin{prp}\label{Criterion}
\begin{enumerate}
\item[$(1)$]
\begin{itemize}
\item[$(Ia)$] Monotonicity(Loewner 1934 \cite{L})
$$
f \in P_n(I) \Longleftrightarrow ([t_i,t_j]) 
\geq 0 \ \hbox{for any}\ \{t_1, t_2, \dots, t_n\}
$$
\item[$(IIa)$] Convexity (Kraus 1936 \cite{K})
$$
f \in K_n(I) \Longleftrightarrow ([t_1, t_i, t_j]) \geq 0 
\ \hbox{for any}\ \{t_1, t_2, \dots, t_n\},
$$
where $t_1$ can be replaced by any (fixed) $t_k$.
\end{itemize}
\item[(2)]
\begin{itemize}
\item[$(Ib)$] Monotonicity (Loewner 1934 \cite{L}, Dobsch 1937 \cite{Do}-Donoghue 1974\cite{Dg})

For $f \in C^{2n-1}(I)$
$$
f \in P_n(I) \Longleftrightarrow M_n(f;t) = 
\left(
\frac{f^{(i+j-1)}(t)}{(i + j - 1)!}
\right) \geq 0 \ \forall t \in I
$$
\item[$(IIb)$] Convexity (Hansen-Tomiyama 2007\cite{HT})
For $f \in C^{2n}(I)$
$$
f \in K_n(I) \Longrightarrow K_n(f;t) = \left(\frac{f^{(i+j)}(t)}{(i+j)!}\right)
\geq 0 \ \forall t \in I.
$$
In particular, for $n = 2$ the converse is also true.
\end{itemize}
\end{enumerate}
\end{prp}

We remind that to prove the implication $M_n(f;t) \geq 0 \Rightarrow f \in P_n(I)$
in $(Ib)$ the local property for the monotonicity plays an essential role.
Similarily to prove the converse implication in the criterion of convexity in 
$(IIb)$ in the above proposition 
we need {\bf the local property conjecture for the convexity}, that is, 
if $f$ is $n$-convex in the intervals $(a, b)$ and $(c, d)$ 
$(a < c < b < d)$, then $f$ is $n$-convex on $(a, d)$.

\vskip 3mm

Now we have only a partial sufficiency, that is, if $K_n(f;t_0)$ is positive, 
then there exists a neighborhood of $t_0$ on which $f$ is $n$-convex. 
(See \cite[Theorem~1.2]{HT} for example.)

\vskip 3mm

Though the method for the implication $(II_b) \Rightarrow (IIa)$ 
under the assumption of the local property theorem for the convexity may be 
familiar for some specialist, we provide here the proof for readers' convenience.

\vskip 3mm

\begin{prp}\label{prp:local property}
Let $f \in C^{2n}(I)$ such that $K_n(f;t) = \left(\frac{f^{(i+j)}(t)}{(i+j)!}\right)
\geq 0 \ \forall t \in I$. 
Suppose that $n$-convexity has the local property. 
Then $f \in K_n(I)$.
\end{prp}

\begin{proof}
If $K_n(f;t) > 0 $ for $t \in I$, for each $t \in [0, \alpha)$
there is an open interval $I_t$ such that  
$f$ is $n$-convex in $I_t$ (c.f. \cite[Theorem~1.2]{HT}). 
Hence for any compact subset $K$ in $I$ there are 
finitely many intervals $I_{t_i} \subset I$ such that 
$K \subset \cup_{i=1}^lI_{t_i}$. From the local property for 
$n$-convexity, we conclude that $f$ is $n$-convex on $K$. 
Taking a sequence of closed intervals  
$K_l$ such that $K_l \subset K_{l+1}$ and $I = \cup_{l=1}^\infty K_l$.
Since the restriction of $f$ to $K_l$ is $n$-convex for each $l$, 
we know that $f$ is $n$-convex on $I$.

In the case that $K_n(f;t)$ is non-negative, we choose a function 
$h$ for which $K_n(h;t) > 0$ as in \cite[Proposition~2.1]{HTII}.
Then for each positive number $\varepsilon > 0$ we have 
$$
K_n(f + \varepsilon h;t) = K_n(f;t) + \varepsilon K_n(h;t) > 0
$$
for all $t \in I$, and we know that 
$f + \varepsilon h $ is $n$-convex from the first observation. 
Hence we can conclude that $f$ is $n$-convex.
\end{proof}

\newpage
\section{The class $Q_n$}

In this secton we introduce the class $Q_n(I)$ on an interval $I$ and 
its characterization from \cite{Dg}.

\vskip 3mm

\begin{dfn}\label{dfn:negative definite}
Let $H^n = \{x = (x_1, \dots, x_n) \in \C^n\mid \sum_{i=1}^nx_i = 0\}$. 
An $n \times n$ Hermitian matrix $A$ is said to be {\it conditionally positive definite} 
(or almost positive) if
$$
(x \mid Ax) \geq 0
$$
for all $x \in H^n$ and {\it conditionally negative definite} if 
$- A$ is conditionally positive definite.
\end{dfn}

\vskip 3mm

\begin{exa}
For $n \in \N$ the matrix $(i+j)_{1\leq i, j, \leq n}$ is conditional positive and 
conditional negative. Indeed, for any $x = (x_1, \dots, x_n) \in H^n$ 
\begin{align*}
(x \mid (i + j)x) &= \sum_{i, j=1}^n(i + j)x_i\overline{x_j}\\
&= \sum_{i, j= 1}^nix_i\overline{x_j} + \sum_{i, j = 1}^njx_i\overline{x_j}\\
&= (\sum_{i=1}^nix_i)\overline{\sum_{j=1}^nx_j} + (\sum_{i=1}^nx_i)\overline{\sum_{j=1}^njx_j}\\
& = 0 + 0 = 0.
\end{align*}
\end{exa}

\vskip 3mm

The following is well known but we put it for readers' convenience.

\begin{lem}\label{lem:conditionally positive definite}
\cite[XV Lemma~1]{Dg}
\cite[Exercise~5.6.16]{Ba2}
Given an $n \times n$ Hermitian matrix $B = [b_{ij}]$ let $D$ be the 
$(n - 1) \times (n - 1)$ matrix with entries 
$$
d_{ij} = b_{ij} + b_{i+1,j+1} - b_{i,j+1} - b_{i+1,j}.
$$
Then $B$ is conditionally positive definite if and only if $D$ is positive semidefinite.
\end{lem}

\vskip 3mm

Let $(a, b)$  be an interval of the real line and $n \in \N$ with $n \geq 2$.
The calss $Q_n(a, b)$ is defined as the class of all real $C^1$ 
functions $f$ on $(a, b)$ such that 
for each $\lambda_1, \lambda_2, \dots, \lambda_n \in (a, b)$ 
the corresponding Lowner matrix $([\lambda_i, \lambda_j]_f)$
is an almost positive matrix. 
Note that $P_n(a, b) \subset Q_n(a, b)$ for each $n \in \N$. 
Since for each $n \in \N$ 
there is an example  of a $n$-monotone and $n$-convex polynomial on $(a, b)$ 
(\cite[Proposition~1.3]{HT}), we know that $Q_n(a, b) \cap K_n(a,b) \not=\emptyset$.
 
\vskip 3mm

The following is a characterization of the class $Q_2(a, b)$.

\vskip 3mm

\begin{lem}\label{lem:Lowner}\cite[XV~Lemma~3]{Dg}
A real $C^1$ function $f$ belongs to $Q_2(a, b)$ if and only if the derivative $f'$ is convex on $(a, b)$.
\end{lem}

\vskip 3mm

The following is characterization of the class $Q_n(a, b)$. 

\begin{lem}\label{lem:Qn}\cite[XV~Lemma~4]{Dg}
A real $C^1$ function $f$ belongs to $Q_n(a, b)$ if and only if 
for every $z \in (a, b)$ the function $[x, z, z]_f$ belongs to $P_{n-1}(a, b)$.
\end{lem}

\vskip 3mm

\newpage
\section{$Q_2(I)$ and the first derivative condition}

We know that for an interval $(a, b)$ 
$\{f'\mid f \in K_2(a, b)\} \subset Q_2(a, b)$.

\vskip 3mm

The following is well-known result, but we give an elementary proof here.

\vskip 3mm

\begin{lem}\label{lem:Cauchy matrix}
The Cauchy matrix $(\frac{1}{i+j})_{1 \leq i, j \leq n}$ is positive definite.
\end{lem}

\vskip 3mm

\begin{proof}
For any $x = (x_1, x_2, \dots, x_n) \in \C^n\backslash\{(0, 0, \dots, 0)\}$ 
\begin{align*}
(x\mid (\frac{1}{i+j})x) &= \sum_{i,j=1}^n\frac{1}{i+j}x_i\overline{x_j}\\
&= \sum_{i,j=1}^n\int_0^1t^{i+j-1}dtx_i\overline{x_j}\\
&= \int_0^1t\sum_{i,j=1}^nt^{i+j-2}x_i\overline{x_j}dt\\
&= \int_0^1t(\sum_{i=1}^nx_it^{i-1})\overline{(\sum_{j=1}^nx_jt^{j-1})}dt\\
&> 0\\
\end{align*}
\end{proof}

\vskip 3mm

\begin{prp}\label{prp:convexToQ_2}
For an interval $(a, b)$ we have the followings.
\begin{enumerate}
\item[$(1)$]
If $f$ is $2$-convex, $f' \in Q_2(a, b)$. 
\item[$(2)$]
If $f'$ is 2-monotone, then $f$ is $2$-convex.
\item[$(3)$] $e^t \in Q_2(a, b)\backslash \{f'\mid f \in K_2(a, b) \cap M_2(a, b)\}$.
\end{enumerate}
\end{prp}

\vskip 3mm

\begin{proof}
$(1)$: 
Let $f$ is 2-convex. Then 

$$
K_2(f,t) = 
\left(
\begin{array}{cc}
\frac{1}{2}f^{(2)}(t)& \frac{1}{6}f^{(3)}(t)\\
\frac{1}{6}f^{(3)}(t)& \frac{1}{24}f^{(4)}(t)
\end {array}
\right)
$$
is positive definite.
Therefore both derivatives $f^{(2)}$ and $f^{(4)}$ are non-negative.
Hence $(f')'$ is convex. 
Therefore $f' \in Q_2(a, b)$ by Lemma~\ref{lem:Lowner}.

$(2)$:
For each $n \in \N$
\begin{align*}
K_n(f;t) &= (\frac{f^{(i+j)}(t)}{(i+j)!})\\
&= (\frac{1}{i+j})\circ (\frac{h^{(i+j-1)}(t)}{(i+j-1)!})\\
&= (\frac{1}{i+j})\circ M_n(h;t),
\end{align*}
where $\circ$ means the Hadamard product of self-adjoint matrices and 
$h = f'$.

Since $f'$ is 2-monotone, $M_2(h;t) \geq 0$ for $t \in [0, \alpha)$. 
Since $(\frac{1}{i+j})$ is positive definite by Lemma~\ref{lem:Cauchy matrix} (\cite[Exercise 1.1.2]{Ba2}) 
and the Hadmard product of 
positive semidfinite matrices becomes positive semidefinite, we have  
$K_2(f;t) \geq 0$ for $t \in [0, \alpha)$.

Hence we could conclude that $f$ is $2$-convex by \cite[Theorem 2.3]{HT}.

$(3)$: Since $e^t$ does not  belong to neither $K_2(a, b)$ nor $M_2(a, b)$, 
$e^t \notin \{f'\mid f \in K_2(a, b) \cap M_2(a, b)\}$.
Obviously $(e^t)' = e^t$ is convex, and $e^t \in Q_2(a, b)$.
\end{proof}

\vskip 3mm

\newpage

\section{Double pilling structure}

In this section we recall the following assertions in \cite{OJ}.

Let $n \in \N$ and 
$f \colon [0, \alpha) \rightarrow \R$ be a continuous function for some $\alpha > 0$.
\begin{enumerate}
\item[$(1)_n$] $f$ is $n$-convex with $f(0) \leq 0$.
\item[$(2)_n$] For an operator $A \in M_n(\C)$ with its spectrum in $[0, \alpha)$ and a contraction $C$
$$
f(C^*AC) \leq C^*f(A)C.
$$
\item[$(3)_n$] $g(t) = \frac{f(t)}{t}$ is $n$-monotone on $(0, \alpha)$.
\end{enumerate}

Then we know that 
$$
(1)_{n+1} \prec (2)_n \sim (3)_n \prec  (1)_{[\frac{n}{2}]}.
$$

Further, in this section we consider the relation between two assertions $(1)_n$ and $(3)_n$.


\begin{prp}\label{prp:2-convex}
Suppose that $f(t)$ is $n$-convex in $[0,\alpha)$ with $f(0) \leq 0$.
Let $g(t) = \frac{f(t)}{t}$. Then
\begin{enumerate} 
\item[$(1)$] 
$K_n(\int g, t)$ is positive semidefinite on $[0, \alpha)$.
\item[$(2)$]
When $n= 2$,  $g(t)$ belongs to the class $Q_2(0,\alpha)$.
\end{enumerate}
\end{prp}

\vskip 3mm

\begin{proof}

$(1)$: 

Since 
$$
f^{(k)}(t) = tg^{(k)}(t) + kg^{(k-1)}(t)
$$
for $1 \leq k$, we have 
$$
t^{(k-1)}f^{(k)}(t) = (t^{(k)}g^{(k-1)}(t))' 
$$
for $2 \leq k$.

Since $f$ is $n$-convex, 
$K_n(f, t) = (\frac{f^{(i+j)}(t)}{(i+j)!}) \geq 0$. 
From the above calculation
\begin{align*}
(t^{i+j-1})\circ K_n(f, t) &= (\frac{t^{i+j-1}f^{(i+j)}(t)}{(i+j)!}) \\
&= (\frac{(t^{i+j}g^{(i+j-1)}(t))'}{(i+j)!}) \geq 0.
\end{align*}

Hence 
\begin{align*}
0 \leq &\int_0^t(\frac{(u^{i+j}g^{(i+j-1)}(u))'}{(i+j)!})du \\
&= (\frac{t^{i+j}g^{i+j-1}(t)}{(i+j)!}) 
- \lim_{s\rightarrow 0}(\frac{s^{i+j}g^{(i+j-1)}(s)}{(i+j)!}).\\
\end{align*}

Since $f(0) \leq 0$ and 
$((-1)^{i+j})$ is positive semidefinite, we have
\begin{align*}
\lim_{s\rightarrow 0}(\frac{s^{i+j}g^{(i+j-1)}(s)}{(i+j)!}) &=
(\frac{(-1)^{i+j-1}f(0)}{i+j})\\
&= (-f(0))(\frac{(-1)^{i+j}}{i+j}) \\
&= - f(0)(\frac{1}{i+j}) \circ ((-1)^{i+j}) \geq 0
\end{align*}
by Lemma~\ref{lem:Cauchy matrix}, where $\circ$ means the 
Hadamard product.

Therefore $(\frac{t^{i+j}g^{(i+j-1)}(t)}{(i+j)!}) \geq 0$.

Next we will check the determinant of the principal submatrix 
$D(k_1, k_2, \dots, k_r) = (\frac{t^{k_i+k_j}g^{k_i+k_j-1}(t)}{(k_i+k_j)!})$ of $M_n(g, t)$ 
for any $k_1, k_2, \dots, k_r \in \{1, 2, \dots, n\}$ and $1 \leq r \leq n$. 
Note that since $(\frac{t^{i+j}g^{(i+j-1)}(t)}{(i+j)!})$ 
is positive semidefinite,
$\det(D(k_1, k_2, \dots, k_r)) \geq 0$ from the standard theorem in 
linear algebras.

\begin{align*}
\det(D(k_1, k_2, \dots, k_r)) 
&= |(\frac{t^{k_i+k_j}g^{k_i+k_j-1}(t)}{(k_i+k_j)!})|\\
&= t^{2k_1}
\left|
\begin{array}{cccc}
\frac{g^{(2k_1-1)}(t)}{(2k_1)!} &\frac{t^{k_2-k_1}g^{(k_1+k_2-1)}(t)}{(k_1+k_2)!}
&\cdots&\frac{t^{k_r-k_1}g^{k_1+k_r-1}(t)}{(k_1+k_r)!}\\
\frac{t^{k_2+k_1}g^{(k_1+k_2-1)}(t)}{(k_2+k_1)!}&\frac{t^{2k_2}g^{(2k_2-1)}(t)}{(2k_2)!}&\cdots
&\frac{t^{k_2+k_r}g^{(k_2+k_r-1)}(t)}{(k_2+k_r)!}\\
\vdots&\vdots&&\vdots\\
\frac{t^{k_r+k_1}g^{(k_1+k_r-1)}(t)}{(k_1+k_r)!}
&\frac{t^{k_r+k_2}g^{(k_2+k_r-1)}(t)}{(k_2+k_r)!}
&\cdots&\frac{t^{2k_r}g^{(2k_r-1)}(t)}{(2k_r)!}\\
\end{array}
\right|\\
&= t^{2k_1}t^{k_1+k_2}
\left|
\begin{array}{cccc}
\frac{g^{(2k_1-1)}(t)}{(2k_1)!} &\frac{t^{k_2-k_1}g^{(k_1+k_2-1)}(t)}{(k_1+k_2)!}
&\cdots&\frac{t^{k_r-k_1}g^{(k_1+k_r-1)}(t)}{(k_1+k_r)!}\\
\frac{g^{(k_2+k_1-1)}(t)}{(k_1+k_2)!}&\frac{t^{k_2-k_1}g^{(2k_2-1)}(t)}{(2k_2)!}&
&\frac{t^{k_r-k_1}g^{(k_2+k_r-1)}(t)}{(k_2+k_r)!}\\
\vdots&\vdots&&\vdots\\
\frac{t^{k_r+k_1}g^{(k_r+k_1-1)}(t)}{(k_r+k_1)!}
&\frac{t^{k_r+k_2}g^{(k_r+k_2-1)}(t)}{(k_r+k_2)!}
&\cdots&\frac{t^{2k_r}g^{(2k_r-1)}(t)}{(2k_r)!}\\
\end{array}
\right|\\
&= \cdots\\
&= t^{2k_1}t^{k_2+k_1} \cdots t^{k_r+k_1}
\left|
\begin{array}{cccc}
\frac{g^{(2k_1-1)}(t)}{(2k_1)!}&\frac{t^{k_2-k_1}g^{(k_1+k_2-1)}(t)}{(k_1+k_2)!}&
\cdots&\frac{t^{k_r-k_1}g^{(k_1+k_r-1)}(t)}{(k_1+k_r)!}\\
\frac{g^{(k_2+k_1-1)}(t)}{(k_2+k_1)!}&\frac{t^{k_2-k_1}g^{(2k_2-1)}(t)}{(2k_2)!}&\cdots&
\frac{t^{k_r-k_1}g^{(k_2+k_r-1)}(t)}{(k_2+k_r)!}\\
\vdots&&&\\
\frac{g^{(k_r+k_1-1)}(t)}{(k_r+k_1)!}&\frac{t^{k_2-k_1}g^{(k_r+k_2-1)}(t)}{(k_r+k_2)!}
&\cdots&\frac{t^{k_r-k_1}g^{(2k_r-1)}(t)}{(2k_r)!}\\
\end{array}
\right|\\
&= (t^{2k_1}t^{k_2+k_1} \cdots t^{k_r+k_1})t^{k_2-k_1}
\left|
\begin{array}{cccc}
\frac{g^{(2k_1-1)}(t)}{(2k_1)!}&\frac{g^{(k_1+k_2-1)}(t)}{(k_1+k_2)!}&
\cdots&\frac{t^{k_r-k_1}g^{(k_1+k_r-1)}(t)}{(k_1+k_r)!}\\
\frac{g^{(k_2+k_1-1)}(t)}{(k_2+k_1)!}&\frac{g^{(2k_2-1)}(t)}{(2k_2)!}&\cdots&
\frac{t^{k_r-k_1}g^{(k_2+k_r-1)}(t)}{(k_2+k_r)!}\\
\vdots&&&\\
\frac{g^{(k_r+k_1-1)}(t)}{(k_r+k_1)!}&\frac{g^{(k_r+k_2-1)}(t)}{(k_r+k_2)!}
&\cdots&\frac{t^{k_r-k_1}g^{(2k_r-1)}(t)}{(2k_r)!}\\
\end{array}
\right|\\
&=\cdots\\
&= (t^{2k_1}t^{k_2+k_1} \cdots t^{k_r+k_1})t^{k_2-k_1}\times\cdots\times t^{k_r-k_1}
\left|
\begin{array}{cccc}
\frac{g^{(2k_1-1)}(t)}{(2k_1)!}&\frac{g^{(k_1+k_2-1)}(t)}{(k_1+k_2)!}&
\cdots&\frac{g^{(k_1+k_r-1)}(t)}{(k_1+k_r)!}\\
\frac{g^{(k_2+k_1-1)}(t)}{(k_2+k_1)!}&\frac{g^{(2k_2-1)}(t)}{(2k_2)!}&\cdots&
\frac{g^{(k_2+k_r-1)}(t)}{(k_2+k_r)!}\\
\vdots&&&\\
\frac{g^{(k_r+k_1-1)}(t)}{(k_r+k_1)!}&\frac{g^{(k_r+k_2-1)}(t)}{(k_r+k_2)!}
&\cdots&\frac{g^{(2k_r-1)}(t)}{(2k_r)!}\\
\end{array}
\right|\\
\end{align*}
This implies that  $\det((\frac{g^{(k_i+k_j-1)}(t)}{(k_i+k_j)!})) \geq 0$.

Since determinants of all principal submatrices of $(\frac{g^{(i+j-1)}(t)}{(i+j)!})$
is nonnegative, we know that
 $(\frac{g^{(i+j-1)}(t)}{(i+j)!})$ is positive semidefinite, hence, 
so is $(\frac{(\int g)^{(i+j)}(t)}{(i+j)!})$. 

$(2)$
From the observation of $(1)$ we have





$$
\left(
\begin{array}{cc}
\frac{1}{2}g'(t)& \frac{1}{6}g^{(2)}(t)\\
\frac{1}{6}g^{(2)}(t)& \frac{1}{24}g^{(3)}(t)
\end{array}
\right)
$$
is positive semidefinite for $t$ in $(0, \alpha)$.

Hence $g^{(3)}(t) \geq 0$ for 
$t \in (0, \alpha)$. 
From Lemma~\ref{lem:Lowner}
we conclude that $g$ belongs to the class $Q_2(0, \alpha)$.
\end{proof}

\vskip 3mm

\begin{rmk}
From Proposition~\ref{prp:2-convex} if $f$ is $k$-convex, we have 
$(\frac{(\int g)^{(i+j)}}{(i+j)!})$ is positive semidefinite. Hence
if the local property theorem for the $k$-convexity is true, then 
we could conclude that $\int g$ is $k$-convex in $(0, \alpha)$.
\hfill$\qed$
\end{rmk}

\vskip 3mm

\begin{rmk}
In Propositions~\ref{prp:2-convex} and \cite[Proposition~2.7]{OJ} 
we can not drop the condition $f(0) \leq 0$.
Indeed, consider $h(t) = - \log(t + 1) + 1$ for $[0, 1)$. 
Then $h$ is 2-convex from 
$$
\det(K_2(h;t)) = \frac{1}{72(t + 1)^6}.
$$

On the contrary, from the similar calculation in 
 Example~\ref{exa:(1)to(3)}.

we have 
\begin{align*} 
(\frac{h(t)}{t})' &= \frac{-t + (t + 1)\log(t + 1) - t^2(t + 1)}{t^2(t + 1)}\\
(\frac{h(t)}{t})^{(2)} &= - \frac{- 3t^2 - 2t + 2t^2\log(t + 1) + 4\log(t + 1) + 2\log(t + 1) 
+ 2(t + 1)^2}{(t + 1)^2t^3}\\
(\frac{h(t)}{t})^{(3)} &= \frac{- 11t^3 - 15t^2 - 6t + 6t^3\log(t + 1) + 18t^2\log(t + 1) + 18t\log(t + 1) 
+ 6\log(t + 1) - 6(t + 1)^3}{(t + 1)^3t^4}\\
\end{align*}
We have then $(\frac{h(t)}{t})^{(3)} < 0$ for some $t \in (0, 1)$.
Hence $(\frac{h(t)}{t})'$ is not convex, that is, 
$\frac{h(t)}{t}$ does not belong to $Q_2(0,1)$ by Lemma~\ref{lem:Lowner}. 
In particular, $\frac{h(t)}{t} \notin P_2(0, 1)$.
\hfill$\qed$
\end{rmk}

\vskip 3mm

\begin{prp}\label{prp:matrix convex}
Let $\alpha > 0$ and $f$ be a $C^2$-function on $[0, \alpha)$.
Let $f(t) = tg(t)$. Then $g$ is a $C^2$-function and 
for any distinct $t_1, t_2, \dots t_n \in (0, \alpha)$ we have
\begin{align*}
[t_1, t_i, t_j]_f = t_1[t_1,t_i,t_j]_g + [t_i,t_j]_g \quad (1 \leq i, j \leq n). 
\end{align*}

Therefore, if $g$ is $n$-monotone and $n$-convex, then $f$ is $n$-convex. 
\end{prp}

\vskip 3mm

\begin{proof}
For any distinct $t_1, t_2, \dots t_n \in (0, \alpha)$ we have
\begin{align*}
[t_1,t_i,t_j]_f &= \frac{[t_1,t_i]_f - [t_i,t_j]_f}{t_1 - t_j}\\
&= \frac{\frac{f(t_1) - f(t_i)}{t_1 - t_i} - \frac{f(t_i) - f(t_j)}{t_i - t_j}}{t_1 - t_j}\\
&= \frac{\frac{t_1g(t_1) - t_ig(t_i)}{t_1 - t_i} - \frac{t_ig(t_i) - t_jg(t_j)}{t_i - t_j}}{t_1 - t_j}\\
&= \frac{\frac{t_1(g(t_1) - g(t_i)) + (t_1 - t_i)g(t_i)}{t_1 - t_i} - 
\frac{(t_i - t_j)g(t_i) - t_j(g(t_j) - g(t_i))}{t_i - t_j}}{t_1 - t_j}\\
&= \frac{\frac{t_1(g(t_1) - g(t_i))}{t_1 - t_i} + g(t_i) - g(t_i) + 
\frac{t_j(g(t_j) - g(t_i))}{t_i - t_j}}{t_1 - t_j}\\
&= \frac{t_1\left\{\frac{g(t_1) - g(t_i)}{t_1 - t_i} - \frac{g(t_i) - g(t_j)}{t_i - t_j}\right\} 
+ t_1\frac{g(t_i) - g(t_j)}{t_i - t_j} - t_j\frac{g(t_i) - g(t_j)}{t_i - t_j}}{t_1 - t_j}\\
&= \frac{t_1\left\{\frac{g(t_1) - g(t_i)}{t_1 - t_i} - \frac{g(t_i) - g(t_j)}{t_i - t_j}\right\}}{t_1 - t_j}
+ \frac{(t_1 - t_j)\frac{g(t_i) - g(t_j)}{t_i - t_j}}{t_1 - t_j}\\
&= t_1[t_1, t_i,t_j]_g + [t_i,t_j]_g
\end{align*}

Suppose that $g$ is $n$-monotone and $n$-convex.
Since $g$ is $n$-monotone, the correspondent 
Loewner matrix $([t_i,t_j]_g)$
is positive semidefinite by \cite{L}. 
Similarily, since $g$ is $n$-convex,  
$([t_1,t_i,t_j]_g)$ is positive semidefinite by \cite{L}.

Therefore, from the above estimate we have 
for any distinct $t_1, t_2, \dots, t_n$ in $(0, \alpha)$
\begin{align*}
([t_1, t_i, t_j]_f)
&= t_1([t_1, t_i,t_j]_g) + ([t_i,t_j]_g)\\
&\geq 0.
\end{align*}
Hence $f$ is $n$-convex by \cite{L}.
\end{proof}

\vskip 3mm

\begin{rmk}
When $f$ is $n$-convex with $f(0) \leq 0$, we can give another proof of \cite[Theorem~2.2]{OJ} 
(that is, $g(t) = \frac{f(t)}{t}$ is $(n - 1)$-positive) 
using Proposition~\ref{prp:matrix convex} as follows:

For any distinct $t_1, t_2, \dots, t_{n-1}, t_n \in (0, \alpha)$
we have 
\begin{align*}
([t_i,t_j]_g)_{1\leq i, j \leq n-1} 
&= ([t_n,t_i,t_j]_f)_{1 \leq i, j \leq n - 1} - (t_n[t_n, t_i, t_j]_g)_{1\leq i, j \leq n - 1} \\
&= ([t_n,t_i,t_j]_f)_{1 \leq i, j \leq n - 1}
- 
(t_n  \frac{1}{(t_n - t_j)}
\{\frac{g(t_n)-g(t_i)}{t_n - t_i} - \frac{g(t_i)-g(t_j)}{t_i - t_j}\})_{1\leq i, j, \leq n-1} 
\\
&= ([t_n,t_i,t_j]_f)_{1 \leq i, j \leq n - 1}
-
(\frac{1}{(t_n - t_j)}
\{t_n\frac{g(t_n)-g(t_i)}{t_n - t_i} - t_n\frac{g(t_i)-g(t_j)}{t_i - t_j}\})_{1\leq i, j, \leq n-1} \\
&= ([t_n,t_i,t_j]_f)_{1 \leq i, j \leq n - 1}
-
(\frac{1}{(t_n - t_j)}\{\frac{f(t_n) - t_ng(t_i)}{t_n - t_i} - t_n 
\frac{g(t_i)-g(t_j)}{t_i-t_j}\})_{1 \leq i, j \leq n - 1} \\
&\geq  - (\frac{1}{(t_n - t_j)}\{ \frac{f(t_n) - t_ng(t_i)}{t_n - t_i} - t_n 
\frac{g(t_i)-g(t_j)}{t_i-t_j}\})_{1 \leq i, j \leq n - 1}
\ (t \in (0, \alpha))
\end{align*}

Note that since $([t_n,t_i,t_j]_f)_{1\leq i, j \leq n}$ is positive semidefinite, 
$([t_n,t_i,t_j]_f)_{1\leq i, j \leq n-1}$ is positive semidefinite.

When $t_n \rightarrow 0$, we get 
\begin{align*}
M_{n-1}(g;t) 
&\geq (\frac{-f(0)}{t_it_j}) \geq 0.
\end{align*}
Hence $g$ is $(n - 1)$-monotone by \cite{L}.
\end{rmk}

\vskip 3mm

In \cite[Proposition 2.6]{OJ} the authors showed that 
for any $2$-convex polynomial $f$ with the degree less than or equal to 5 
on $[0, \alpha)$
$g(t) = \frac{f(t)}{t}$ is $2$-monotone. 

\vskip 3mm


\vskip 3mm 

We have another affirmative example for the implication from $(1)$ to $(3)$.

\vskip 3mm

\begin{exa}\label{exa:(1)to(3)}
Let $f(t) = - \log(t + 1)$ for $[0, 1)$. Then 
$f$ is $2$-convex. 
Indeed, since for any $n \in \N$ $f^{(n)}( t) = (- 1)^n\frac{(n - 1)!}{(t + 1)^n}$, 
we have 
\begin{align*}
K_2(f, t) &= (\frac{f^{(i+j)}(t)}{(i + j)!})\\
\det(K_2(f, t)) &= \frac{1}{72(t + 1)^6}.
\end{align*}
Since $f^{(2)}(t) \geq 0$, $f^{(4)}(t) \geq 0$,  and $\det(K_2(f,t)) \geq 0$ for $t \in [0, 1)$,  
$K_2(f, t)$ is positive semidefinite for $t \in [0, 1)$. 
Therefore, $f$ is $2$-convex by \cite{HT}. 

Let $g(t) = \frac{f(t)}{t} = - \frac{\log(t + 1)}{t}$. Then we will show that 
$g$ is $2$-monotone on $(0, 1)$.
We have, then,  

\begin{align*}
g'(t) &= \frac{-t + (t + 1)\log(t + 1)}{t^2(t + 1)}\\
g^{(2)}(t) &= - \frac{- 3t^2 - 2t + 2t^2\log(t + 1) + 4t\log(t + 1) + 2\log(t + 1)}{(t + 1)^2t^3}\\
g^{(3)}(t) &= \frac{- 11t^3 - 15t^2 - 6t + 6t^3\log(t + 1) + 18t^2\log(t + 1) + 18t\log(t + 1) + 6\log(t + 1)}{(t + 1)^3t^4}\\
\det(M_2(g, t)) &= 12 \frac{- 5t^2 - 6t + 2t^2\log(t + 1) + 8t \log(t + 1) + 6\log(t + 1)}{t^4(t + 1)^4}.
\end{align*}

From the simple calculations we conclude that all $g', g^{(3)}, \det(M_2(g, t))$ are nonnegative.
Note that $(t+1)^3t^4g^{(3)}(t) = \frac{1}{10}t^6 - \frac{3}{10}t^5 + \frac{3}{2}t^4 + O(t^7)$
and $t^4(t+1)^4\det(M_2(g,t)) = \frac{1}{6}t^4 - \frac{2}{15}t^5 + \frac{1}{10}t^6 + O(t^7)$ by 
using the Maclaurin series of $\log(t + 1)$.

Hence $M_2(g, t)$ is positive semidefinite. Therefore, $g$ is $2$-monotone.
\end{exa}

Generally,  however, even if $n = 2$, 
the assertion $(1)$ does not necesarily imply the assertion $(3)$. 

We need the following simple observation.

\vskip 3mm

\begin{lem}\label{lem:parallelism movement}
For $\alpha, \beta > 0$ we define a function $h \colon [0, \alpha) \rightarrow [0,\beta)$
by $h(t) = \frac{\beta}{\alpha}t$. Then
\begin{enumerate}
\item
If $f \in K_n([0, \alpha))$, then $f \circ h^{-1} \in K_n([0,\beta))$.
\item
If $\frac{f(t)}{t} \in M_n((0, \alpha))$, then $\frac{(f\circ h^{-1})(t)}{t} \in M_n((0, \beta))$.
\end{enumerate}
\end{lem}

\begin{proof}
$(1)$: 
Let $A, B \in M_n(\C)$ with $0 \leq A, B \leq \beta I_n$, where $I_n$ is the identity matrix 
in $M_n(\C)$. Take for any $\lambda \in [0, 1]$. 
Since $f \in K_n([0, \alpha))$ and 
$0 \leq \frac{\alpha}{\beta}A, \frac{\alpha}{\beta}B \leq \alpha I_n$, we have 
\begin{align*}
f(\lambda \frac{\alpha}{\beta}A + (1 - \lambda)\frac{\alpha}{\beta}B)
&\leq \lambda f(\frac{\alpha}{\beta}A) + (1 - \lambda)f(\frac{\alpha}{\beta}B)\\
f(\frac{\alpha}{\beta}(\lambda A + (1 - \lambda)B) &\leq 
\lambda f(\frac{\alpha}{\beta}A) + (1 - \lambda)f(\frac{\alpha}{\beta}B)\\
(f \circ h^{-1})(\lambda A + (1 - \lambda)B) 
&\leq \lambda (f \circ h^{-1})(A) + (1 - \lambda)(f \circ h^{-1})(B)
\end{align*}

Hence $f \circ h^{-1} \in K_n([0, \beta))$.

$(2)$:
Let $A, B \in M_n(\C)$ with $0 \leq A \leq B \leq \beta I_n$. 
Since $\frac{f(t)}{t} \in M_n((0, \alpha))$ and 
$0 \leq \frac{\alpha}{\beta}A \leq \frac{\alpha}{\beta}B \leq \alpha I_n$, 
we have 
\begin{align*}
&0 \leq (\frac{\alpha}{\beta}A)^{-1}f(\frac{\alpha}{\beta}A) 
\leq (\frac{\alpha}{\beta}B)^{-1}f(\frac{\alpha}{\beta}B) \\
&0 \leq \frac{\beta}{\alpha}A^{-1}f(\frac{\alpha}{\beta}A) 
\leq \frac{\beta}{\alpha}B^{-1}f(\frac{\alpha}{\beta}B)\\
&0 \leq \frac{\beta}{\alpha}A^{-1}(f\circ h^{-1})(A) 
\leq \frac{\beta}{\alpha}B^{-1}(f\circ h^{-1})(B)\\
&0 \leq A^{-1}(f\circ h^{-1})(A) \leq B^{-1}(f \circ h^{-1})(B).
\end{align*}

Hence $\frac{(f \circ h^{-1})(t)}{t} \in M_n((0, \beta))$.
\end{proof}

\vskip 3mm

\begin{thm}\label{thm:counterexample (1)->(3)}
 For any $\alpha > 0$ 
there is a $2$-convex function $f$ on $[0, \alpha)$, 
but $g(t) = \frac{f(t)}{t}$ is not $2$-monotone on $(0, \alpha)$.
\end{thm}

\begin{proof}
Let $f(t) =  t + \frac{1}{2}t^2 + \frac{1}{3}t^3 - \log(t + 1)$ 
and $g(t) = \frac{f(t)}{t} = 1 + \frac{1}{2}t + \frac{1}{3}t^2 - \frac{\log(t + 1)}{t}$. 

Since 
\begin{align*}
K_2(f, t) 
&= 
\left(\begin{array}{cc}
\frac{1}{2} + t + \frac{1}{2(t + 1)^2}&\frac{1}{3} - \frac{1}{3(t + 1)^3}\\
\frac{1}{3} - \frac{1}{3(t + 1)^3}& \frac{1}{4(t + 1)^4}
\end{array}
\right)\\
\det(K_2(f,t)) &= 
- \frac{1}{72}\frac{1}{(t + 1)^6}(27t^2 - 36t - 18 + 126t^3 + 
8t^6 + 48t^5 + 120t^4), \\
\end{align*}
$K_2(f, t)$ is positive definite for some $[0, \beta)$ $(\alpha > 0)$. 
For example $\beta = 0.1$.

On the contrary,

\begin{align*}
M_2(g,t) &= 
\left(\begin{array}{cc}
\frac{1}{2} + \frac{2}{3}t - \frac{1}{(t+1)t} + \frac{\log(t + 1)}{t^2}&
\frac{1}{3} + \frac{1}{2(t + 1)^2t} + \frac{1}{(t + 1)t^2} - \frac{\log(t + 1)}{t^3}\\
\frac{1}{3} + \frac{1}{2(t + 1)^2t} + \frac{1}{(t + 1)t^2} - \frac{\log(t + 1)}{t^3}&
- \frac{1}{3(t + 1)^3t} - \frac{1}{2(t + 1)^2t^2} - \frac{1}{(t + 1)t^3} + \frac{\log(t + 1)}{t^4}
\end{array}
\right)\\
\det(M_2(g,t)) &= \frac{1}{36}\frac{1}{t^4(t + 1)^4}
\{- 4t^8 - 16t^7 - 24t^6 - 96t^5 - 237t^4 - 246t^3 - 126t^2 - 36t \\
&+ 48t^5\log(t + 1) + 210t^4\log(t + 1) + 360t^3\log(t + 1) + 306t^2\log(t + 1) 
+ 144t\log(t + 1) + 36\log(t + 1)\}\\
\end{align*}
Since $t > \log(t + 1)$, there must exist a sufficiently small $t_0 \in (0, \beta)$ such that 
$\det(M_2(g,t_0)) < 0$. 
For example $\det(M_2(g,t))(1.0\times 10^{-9}) = -2.7777778682\times 10^{17}$ 
using Maple.

Therefore, $f$ is $2$-convex on $[0, \beta)$, but $g$ is not $2$-monotone on $(0, \beta)$.

Let $h\colon [0, \alpha) \rightarrow [0, \beta)$  by $h(t) = \frac{\beta}{\alpha}t$
and set $F = f \circ h$. Then $F \in K_2([0, \alpha)$. 
Suppose that $\frac{F(t)}{t}$ is $2$-monotone in $(0, \alpha)$. 
Then by Lemma~\ref{lem:parallelism movement}$(2)$ $\frac{f(t)}{t} (= g(t))$ is $2$-monotone on $(0, \beta)$. 
This is a contradiction. Hence $\frac{F(t)}{t} \notin M_2((0, \alpha))$.
\end{proof}




\vskip 3mm

Next we give an affirmative answer for the convex implication 
from $(3)_n$ to $(1)_n$.

\vskip 3mm

\begin{lem}\label{lem:divided}
Let $f$ be a $C^2$ function on some interval $[0, \alpha)$.
Let take $t \in [0, \alpha)$ and consider a function $[t, z]_f$.
Then
$$
\frac{d}{dz}[t, z]_f = [t, z, z]_f.
$$
\end{lem}

\vskip 3mm

\begin{proof}
Direct computation.
\end{proof}

\vskip 3mm

\begin{thm}\label{thm:Q_n}
Let $f \in \Q_n([0, \alpha))$. 
Suppose that $\frac{f(t) - f(0)}{t}$ is $(n - 1)$-monotone on $(0, \alpha)$.
Then $f$ is $(n - 1)$-convex. 
In particular, if $f(0) = 0$, then $(n - 1)$-monotonicity of $\frac{f(t)}{t}$ 
implies $(n - 1)$-convexity of $f$.
\end{thm}

\vskip 3mm

\begin{proof}
Since $f \in Q_n([0, \alpha))$, for any $z \in [0, \alpha)$
$h(t) = [t,z, z]$ is (n - 1)-monotone \cite[XV Lemma~4]{Dg}. 
Hence for any $\lambda_1, \dots, \lambda_{n-1} \in [0, \alpha)$
$([\lambda_i, \lambda_j]_h) \geq 0$, that is, for any $\xi = (\xi_1, \xi_2, \dots, \xi_{n-1}) \in \C^{n-1}$
\begin{align*}
(([\lambda_i, \lambda_j]_h)\xi \mid \xi) &\geq 0, \\
\sum_{i,j=1}^{n-1}[\lambda_i, \lambda_j]_h\xi_j\overline{\xi_i} \geq 0.
\end{align*}
Since
\begin{align*}
\sum_{i,j=1}^{n-1}[\lambda_i, \lambda_j]_h\xi_j\overline{\xi_i} 
&= \sum_{i,j=1}^{n-1}\frac{1}{\lambda_i - \lambda_j}(h(\lambda_i) - h(\lambda_j))\xi_j\overline{\xi_i}\\
&= \sum_{i,j=1}^{n-1}\frac{1}{\lambda_i - \lambda_j}
([\lambda_i,z,z]_f - [\lambda_j,z,z]_f\xi_j\overline{\xi_i}\\
&= \sum_{i,j=1}^{n-1}\frac{1}{\lambda_i - \lambda_j}
(\frac{d}{dz}[\lambda_i, z]_f - \frac{d}{dz}[\lambda_j, z]_f)\xi_j\overline{\xi_i}
\ \quad (\hbox{Lemma~\ref{lem:divided}}),\\
\end{align*} 
we have then 
\begin{align*}
0 &\leq \int_0^z(\sum_{i,j=1}^{n-1}\frac{1}{\lambda_i - \lambda_j}
(\frac{d}{dz}[\lambda_i, z]_f - \frac{d}{dz}[\lambda_j, z]_f)\xi_j\overline{\xi_i})dz\\
&= \sum_{i,j=1}^{n-1}\frac{1}{\lambda_i - \lambda_j}
\{([\lambda_i, z]_f - [\lambda_i, 0]_f) - ([\lambda_j, z]_f- [\lambda_j,0]_f)\}\xi_j\overline{\xi_i}\\
&= \sum_{i,j=1}^{n-1}\frac{1}{\lambda_i - \lambda_j}
\{([\lambda_i, z]_f - [\lambda_j, z]_f\}\xi_j\overline{\xi_i}) 
- ([\lambda_i, 0]_f - [\lambda_j, 0]_f\}\xi_j\overline{\xi_i})\}\\
&= \sum_{i,j=1}^{n-1}[\lambda_i,\lambda_j,z]_f\xi_j\overline{\xi_i} - 
\sum_{i,j=1}^{n-1}[\lambda_i,\lambda_j,0]_f\xi_j\overline{\xi_i}.
\end{align*}

Hence, 
$([\lambda_i,\lambda_j,0]_f]) \leq ([\lambda_i,\lambda_j,z]_f)$. 
Note that $([\lambda_i,\lambda_j,0]_f]) = ([\lambda_i, \lambda_j]_{\frac{f(t) - f0)}{t}})$.

Since $\frac{f(t)-f(0)}{t}$ is $(n - 1)$-monotone by the assumption, 
$([\lambda_i, \lambda_j]_{\frac{f(t) - f0)}{t}})$ is positive semidefinite, and 
so is $([\lambda_i,\lambda_j,0]_f])$.
Hence 
$([\lambda_i,\lambda_j,z]_f)$ is positive semidefinite.
Since $z \in [0, \alpha)$ is an arbitrary, 
$([\lambda_i, \lambda_j,\lambda_{n-1}]_f)$ is positive semidefinite, 
that is,  $f$ is $(n - 1)$-convex by \cite{K}.
\end{proof}

\vskip 3mm

\begin{rmk}
From \cite[Lemma~2.3(2)]{OJ} if $f(0) \leq 0$, the $n$-monotonicity of $\frac{f(t) - f(0)}{t}$
implies the $n$-monotonicity of $\frac{f(t)}{t}$. But it is not obvious that the converse 
is true. 

Suppose that $\frac{f(t)}{t}$ is operator monotone. We have then
\begin{align*}
f \ \hbox{is operator convex} &\Rightarrow 
f(t) - f(0) \ \hbox{is operator convex}\\
&\Rightarrow \frac{f(t) - f(0)}{t}
\end{align*}
is operator monotone by \cite{HP}.

Using the same idea we could conclude that 
the $n$-monotonicity and $n$-convexity of $\frac{f(t)}{t}$ implies $(n - 1)$-monotonicity 
of $\frac{f(t) - f(0)}{t}$ by Proposition~\ref{prp:matrix convex} and \cite[Theorem~2.2]{OJ}.
Indeed, if $\frac{f(t)}{t}$ is $n$-monotone and $n$-convex, then $f$ is $n$-convex by 
Proposition~\ref{prp:matrix convex}. Hence $f(t) - f(0)$ is $n$-convex. Therefore, 
$\frac{f(t) - f(0)}{t}$ is $(n - 1)$-monotone by \cite[Theorem~2.2]{OJ}.
\end{rmk}

\vskip 3mm

We obtain the following characterization of $n$-monotonicity of $\frac{f(t) - f(0)}{t}$ 
from the $n$-monotonicity of the first derivative of $f$.

\vskip 3mm

\begin{thm}\label{thm:(3)}
Let $f\colon [0, \alpha) \rightarrow \R$ be a $C^{1}$ function. 
Suppose that $f'$ is $n$-monotone. Then $\frac{f(t) - f(0)}{t}$ is $n$-monotone on 
$(0, \alpha)$.
\end{thm}

\vskip 3mm

\begin{proof}
Note that if $g \colon [0, \alpha) \rightarrow \R$ is $n$-monotone, then for any $u \in [0, 1]$
$h(t) = g(ut)$ is $n$-monotone on $[0, \alpha)$.

Since $f'$ is $n$-monotone, for any 
$\lambda_1, \lambda_2, \dots, \lambda_n \in (0, \alpha)$ 
$([\lambda_i, \lambda_j]_{f'})$ is positive semidefinite, that is, 
for any $\xi_1, \xi_2, \dots, \xi_n \in \C$ 
$$
\sum_{i,j=1}^n\frac{1}{\lambda_i - \lambda_j}
\{f'(\lambda_i) - f'(\lambda_j)\}\xi_j\overline{\xi_i} \geq 0.
$$
Hence for any $u \in [0, 1]$
$$
\sum_{i,j=1}^n\frac{1}{\lambda_i - \lambda_j}
\{f'(u\lambda_i) - f'(u\lambda_j)\}\xi_j\overline{\xi_i} \geq 0.
$$
We have then for any $u \in [0, 1]$ 
\begin{align*}
0 &\leq \int_0^1\sum_{i,j=1}^n\frac{1}{\lambda_i - \lambda_j}\{f'(u\lambda_i) - f'(u\lambda_j)\}
\xi_j\overline{\xi_i}du\\
&= \sum_{i,j=1}^n\frac{1}{\lambda_i - \lambda_j}
\int_0^1\left\{\frac{1}{\lambda_i}f(u\lambda_i) - \frac{1}{\lambda_j}f(u\lambda_j)\right\}'du 
\xi_j\overline{\xi_i}\\
&= \sum_{i,j=1}^n\frac{1}{\lambda_i - \lambda_j}
\left\{\frac{f(\lambda_i) - f(0)}{\lambda_i}
- \frac{f(\lambda_j) - f(0)}{\lambda_j}\right\}
 \xi_j\overline{\xi_i}
\end{align*}
Therefore, $[[\lambda_i, \lambda_j]_{\frac{f(t) - f(0)}{t}}]$ 
is positive semidefinite, and $\frac{f(t) - f(0)}{t}$ is $n$-monotone.
\end{proof}

\vskip 3mm

We have, then,  the following relation 
between the class $Q_n([0, \alpha))$ and the class $K_{n-1}([0, \alpha))$.

\vskip 3mm

\begin{cor}
Let $f \in Q_n([0, \alpha))$. 
Suppose that $f'$ is $(n - 1)$-monotone. Then $f$ is $(n - 1)$-convex.
\end{cor}

\vskip 3mm

\begin{proof}
From Theorem~\ref{thm:(3)} $\frac{f(t) - f(0)}{t}$ is $(n -1)$-monotone 
on $(0, \alpha)$. Hence $f$ is $(n - 1)$-convex by Theorem~\ref{thm:Q_n}.
\end{proof}

\vskip 3mm

\begin{rmk}
It is not true that $Q_n(I) \subset K_{n-1}(I)$  for an interval in 
$[0, \infty)$.
Indeed, 
If $0 < \alpha < 1$, $t^\alpha$ is $n$-monotone, but not
$2$-convex. Hence if $n \geq 3$ and $\beta > 0$, then 
$t^\alpha \in Q_n([0, \beta))$, but $t^\alpha \notin K_{n-1}([0, \beta))$.
\end{rmk}

\vskip 3mm

As pointed out in Proposition~3.5 in \cite{OJ}
the implication from $(3)$ to $(1)$ is not true even if $n = 1$. 
We then have another observation as in Theorem~\ref{thm:counterexample (1)->(3)} when $n = 2$
using Maple.

\vskip 3mm

\begin{thm}\label{thm:counterexample (3) ->(1)}
For any $\alpha > 0$ 
there is a $2$-monotone function $g$ on $(0, \alpha)$, 
but $f(t) = tg(t)$ is not $2$-convex on $[0, \alpha)$.
\end{thm}

\begin{proof}
Let $f$ be a $2$-convex function defined by 
$f(t) = t + \frac{1}{2}t^2 + \frac{1}{3}t^3 + \frac{1}{4}t^4 + \frac{1}{5}t^5$.
Then we know that $K_2(f,t) = \frac{1}{72} + \frac{1}{12}t - \frac{23}{24}t^2 - 2t^3 - 2t^4$ 
is positive definite on $[0, 0.14)$, but 
negative definite on $[0.15, 1)$.

On the contrary, $g(t) = \frac{f(t)}{t} = 1 + \frac{1}{2}t + \frac{1}{3}t^2 
+ \frac{1}{4}t^4 + \frac{1}{5}t^4$ is $2$-monotone on $[0, 0.17]$ from which 
$M_2(g, t) = \frac{1}{72} + \frac{1}{15}t - \frac{77}{120}t^2 - t^3 
- \frac{4}{5}t^4$ is positive definite on $[0, 0.17]$.

Hence $\frac{f(t)}{t} \in M_2([0, 0.17])$, but $f \notin K_2([0, 0.17])$.

Let $\alpha > 0$ and $\beta = 0.17$, and define $h \colon [0, \alpha) \rightarrow [0, \beta)$
by $h(t) = \frac{\beta}{\alpha}t$. Define $G(t) = \frac{(f\circ h)(t)}{t}$. 
Then $G \in M_2([0, \alpha))$ by Lemma~\ref{lem:parallelism movement}$(2)$. 
As in the proof of Theorem~\ref{thm:Q_n}, however, $tG(t) \not\in K_2([0, \alpha))$.
Indeed, suppose that $tG(t) \in K_2([0, \alpha))$. Then $(f \circ h) \in K_2([0, \alpha))$. 
By Lemma~\ref{lem:parallelism movement}$(1)$ we know that $f \in K_2([0, \beta))$, and a 
contradiction. Hence  $tG(t) \not\in K_2([0, \alpha))$.
\end{proof}

\vskip 3mm

Before closing this note we summarize several observations  and a 
problem between condition $(1)$ and condition $(3)$, which are presented at 
the first part in this section.

\vskip 3mm

\begin{thm}\label{thm:summarize}
Let $0 < \alpha \leq \infty$ and $f$ be a real valued function in $[0, \alpha)$ 
with $f(0) \leq 0$. Let $g(t) = \frac{f(t)}{t}$.
Suppose that $f$ is a $C^2$-function. 
\begin{enumerate}
\item[(i)]
If $g$ is $n$-monotone and $n$-convex, then $f$ is $n$-convex.
\item[(ii)]
If $f$ is $2$-convex, 
then $g \in Q_2(0, \alpha)$, but it does not necessarily imply that $g$
is  $2$-monotone.
\item[(iii)]
If $f$ in $Q_{n+1}([0, \alpha))$ with $f(0) = 0$ and $g$ is $n$-monotone, 
then $f$ is $n$-convex. 
In particular, if $f$ is $(n + 1)$-monotone, 
then the implication from $(3)_n$ to $(1)_n$ holds.
\end{enumerate}
\end{thm}


\end{document}